\documentclass[reqno]{amsart}
\usepackage{amssymb,amsmath,amsthm}
\usepackage{geometry}
\numberwithin{equation}{section}
\newtheorem{theorem}{Theorem}[section]
\newtheorem{lemma}[theorem]{Lemma}
\newtheorem{k-lemma}[theorem]{Key lemma}

\newtheorem {proposition}[theorem]{Proposition}
\theoremstyle{definition}

\theoremstyle{remark}
\newtheorem{remark}[theorem]{Remark}

\renewcommand{\O}{\mathbb O}  \renewcommand{\H}{\mathbb H} \newcommand{\C}{\mathbb C}    \newcommand{\R}{\mathbb R}

\begin{document}

\title{An integral transform connecting spherical analysis on harmonic NA groups to that of odd dimensional real hyperbolic spaces}
 \author{A. INTISSAR  $^\dagger$} 
\author{ M.V. OULD MOUSTAPHA $^\ast$} \email{mohamedvall.ouldmoustapha230@gmail.com}
\author{ Z. MOUHCINE $^\dagger$} \email{zakariyaemouhcine@gmail.com}
\address{$^\ast$Department of Mathematics, College of Science and Arts, University of Aljouf, El-Qurayat, Saudi Arabia }
\address{$^\dagger$,
	Laboratory of Analysis and Applications - URAC/03,
	Department of Mathematics, P.O. Box 1014,  Faculty of Sciences,
	Mohammed V University in Rabat, Morocco} 

\keywords{Harmonic $NA$ groups; hyperbolic spaces; Laplace-Beltrami operator; Resolvent kernel; Gauss Hypergeometric function; Jacobi equation; hyperbolic functions.}
	\subjclass[2010]{44A20, 33C05, 53C35}
	\dedicatory{In memory of Professeur Ahmed Intissar (1952-2017).}

\begin{abstract}
The main aim of the present paper is to establish an integral transform connecting spherical analysis on harmonic NA groups to that of odd dimensional real hyperbolic spaces. Moreover, certain interesting integral identities for the Gauss hypergeometric functions  have also been given.
\end{abstract}
\maketitle

\section{Introduction and main results}

Harmonic $N A$ groups and analysis theorem have been studied by several autors \cite[...]{Anker, Berndt, Damek-Ricci, Ricci}. Recall that, as Riemannian manifolds, these solvable Lie groups include all symmetric spaces of noncompact type and rank one, namely the hyperbolic space $H^{n}_{\mathbb F} \, (\mathbb F=\R, \C, \H)$ and $H^{2}_{\O}$, but that most of them are not symmetric, providing numerous counterexamples to the Lichnerowicz conjecture \cite{Lichnerowicz}.
Despite the lack of symmetry, spherical analysis i.e. the analysis of radial functions on these spaces is quite similar to the hyperbolic space case. We shall emphasize that spherical analysis is again a particular case of the Jacobi function analysis  \cite{Anker}.

Let $X=G/K=NA$ be a rank one symmetric space of non compact type where $G$ is a non compact rank one semi-simple group. $K$ is the maximal compact subgroup of $G$ whereas $N$ and $A=\R$ are respectively
the nilpotent and abelian parts that enter into the Iwasawa decomposition of $G=NAK$. Let $Z_X$ denote
the center of $N_X:=N$. $N$ is abelian for real hyperbolic spaces $G/K=H^{n}_{\R}$ and of Heisenberg type (see \cite{Anker, Kaplan, Kaplan2, Korany1, Rouviere}  for the general theory on Heisenberg type groups) in the other cases  for $G/K= H^{n}_{\C},  H^{n}_{\H}$ or $ H^{n}_{\O}$. 

The space $X=G/K=NA$ is a homogeneous  Riemannian space and if $L_X$ denote the Laplace-Beltrami operator of $X=NA$, then the radial part of $L_X$  in geodesic polar coordinates is given by
\begin{equation}\label{1}
rad(L_X)=\frac{d^2}{dr^2}+\left\{ \dim N_X\coth (r) + \dim Z_X\tanh (r) \right\}\frac{d}{dr} , 
\end{equation}
where $r=d(x,y)$ is the geodesic distance between two points $x,y\in X=NA$. 
 
It is well known that the spherical resolvent kernel $R_X(\lambda;x,y):=R_X(\lambda;r)$ of the Laplacian $L_X$ on $X$ can be described as the singular solution at $r=0$ of the following equation of Jacobi type
\begin{equation}\label{2}
   \left(rad(L_X)+\sigma_X^2+\lambda^2\right)R_X(\lambda,r)=0, \quad  r>0, \, \lambda\in \C, 
\end{equation}
where $\sigma_X =:\sigma =\frac{dim N_X+dim Z_X}{2}$ and $\lambda$ is a complex number such that $\Im m \lambda\geq 0.$

In fact the resolvent kernel function of the shifted Laplacian $L_X+\sigma^2_X$ on $X$ is well known to be given in terms of the Gauss Hypergeometric functions $_2F_1(a,b;c;z)$ \cite{Anker, Wallach}. More precisely,  we have  
\begin{equation}\label{resol-X}
   R_X(\lambda,r)=C_X(\lambda)\cosh^{-\sigma+i\lambda}(r) \,\, _2F_1 \left(\frac{\sigma-i\lambda}{2},\frac{\sigma-i\lambda}{2}-\beta; 1-i\lambda; \cosh^{-2}(r)\right)
\end{equation}
for $\Im m\lambda \geq 0$, $\beta=\frac{dim Z_X-1}{2}$ and the constant $C_X(\lambda)$ is given explicitly by
\begin{equation*}\label{cste-X}
C_X(\lambda)= \pi^{-(dim N_X+1)/2}\frac{\Gamma((\sigma-i\lambda)/2) \Gamma((\sigma-i\lambda)/2-\beta)}{4\Gamma(1-i\lambda)},
\end{equation*}
where the hypergeometric function $_2F_1(a,b;c;z)$ reads as
\begin{equation*}
_2F_1(a,b;c;z) = {\Gamma(c)\over \Gamma(a)\Gamma(b)} \sum_{k=0}^\infty {\Gamma(a+k)\Gamma(b+k)\over k!\Gamma(c+k)}  \, z^k 
\end{equation*}
and where $\Gamma(z)=\int_{0}^{\infty} t^{z-1}e^{-t}dt$ is Euler's Gamma-function.
\begin{remark}
i) For instance if $Y$ is a real hyperbolic space the resolvent kernel of the Laplacian operator  $L_Y$ is given
as in (\ref{resol-hyp}) with $\beta=-1/2$ and $\sigma_Y= \dim N_Y/2=(\dim Y-1)/2$, in particular where $\dim Y$ is odd, the resolvent kernel $R_Y(\lambda,r)$ can be given in terms of elementary function, see Section 2.

ii) At this point one should observe that for $X=NA$ with $X$ not a real hyperbolic space, the number $\sigma_X=\frac{dim N_X+dim Z_X}{2}$ is an integer positive number. This occurs even if $X=NA$ is not a symmetric space but of Damek-Ricci spaces type \cite{Rouviere}, i.e, $X=NA$ is a Harmonic group of rank one and the corresponding resolvent kernel $R_X(\lambda,r)=R_{NA}(\lambda,r)$ is given, for example, as in \cite{Will}.
\end{remark}

An integral transform relating the heat kernels on even dimensional hyperbolic spaces to the ones of odd dimensional was obtained by Davies and Mandouvalos \cite{Davies}. More precisely, in \cite{Davies}, it has been established a recurrence relations relating  the heat kernel, written in terms of the hyperbolic distance, on the real hyperbolic spaces $H^{n+1}$ and $H^{n+2}$ of dimension $n+1$ and $n+2$ respectively,  given by the following  integral transform 
\begin{align}
K_{n+1}(t,\rho)= \frac{1}{\sqrt{2}} \int_{\rho}^{\infty} \frac{e^{(2n+1)t/4}}{ (\cosh \mu-\cosh\rho)^{1/2}} K_{n+2}(t,\mu) \sinh(\mu) d\mu.
\end{align}

The main objective of this paper is to show that all resolvent kernels of harmonic $NA$ groups  (including Riemannian symmetric space of non compact type and of rank one) can be expressed as an integral transform of those odd dimensional hyperbolic spaces. Namely the result to which is aimed this paper is to establish the following integral transform 
\begin{theorem}\label{integral transform} Let $X=NA$ be a harmonic group, including rank one symmetric spaces of non compact type
(Real hyperbolic spaces are disregarded). Then the resolvent kernel $R_X(\lambda,r)$ of the Laplacian $L_X$ as given in (\ref{resol-X}) can be expressed in terms of the resolvent kernel of odd dimensional hyperbolic spaces as follows
\begin{equation}\label{integraltransform}
R_X(\lambda,r)=\int_r^\infty W_X(r,\rho)R_Y(\lambda,\rho)\sinh(\rho) d\rho
\end{equation}
where the kernel $W_X(r,\rho)$ is independent of $\lambda\in\C$ and it is given by
\begin{equation}\label{kernel}
W_X(r,\rho) = \frac{2\pi^{\frac{\dim Z_X}{2}}}{\Gamma(\dim Z_X/2)} \cosh^{(1-\dim Z_X)}(r)\left(\cosh^2 \rho-\cosh^2 r\right)^{(\dim Z_X-2)/2}
\end{equation}
while
$R_Y(\lambda,\rho)$ is the resolvent kernel of some odd dimensional hyperbolic spaces $Y$ such that\\
$\dim Y= \dim N_X+ \dim Z_X+1$.
\end{theorem}

The plan of the article is as follows. Inspired by the work of  Davies-Mandouvalos  \cite{Davies}, in Section 2, we have obtained a recurrence relation relating the real hyperbolic resolvent kernels for different dimensions. The Section 3 deals with some integral identities for the hypergeometric functions $_2F_1$.
In the Section 4, we establish an integral transform linking spherical analysis on harmonic NA groups to that of odd dimensional real hyperbolic spaces. Finally, in Section 5, some concluding remarks are made.

\section{Spherical analysis  on the hyperbolic space}

Spherical analysis on hyperbolic spaces was developed in \cite[...]{Davies, Helgason, Koornwinder}. We recall briefly some of it in this section. 
The radial part (in geodesic polar coordinates) of the Laplace-Beltrami operator $L_Y$ on a hyperbolic space $Y$ of dimension $n$ reads as
\begin{align}\label{Lap-hyp}
rad(L_Y) = \frac{d^2}{dr^2}+ (n-1) \coth (r) \frac{d}{dr}, \qquad r>0.
\end{align}

Let $\lambda\in \C$, the spherical functions $\varphi_{Y}(\lambda,.)$ on $Y$ are normalized radial eigenfunctions of $L_Y$ with eigenvalue $\nu=-(\sigma_{Y}^2 + \lambda^2)$ and $\varphi_{Y}(\lambda,0)= 1$. We then have, that $\varphi_{Y}(\lambda,.)$ is the solution of the following differential equation
\begin{align}\label{hyp-equa}
\left( \frac{d^2}{dr^2}+ (n-1) \coth (r) \frac{d}{dr} + ((n-1)/2)^2 + \lambda^2 \right) \varphi_{Y}(\lambda,r) = 0; \quad r>0, \lambda \in \C,
\end{align}
continuous at $r = 0$, and since $\lim_{r\rightarrow 0^+} r \coth(r) = 1$, this equation has a regular singular point at $r =0$. 
This is a Jacobi equation with parameters $\lambda, \alpha = \frac{n-2}{2}$ and $\beta=\frac{-1}{2}$ (for more
group theoretic interpretations of Jacobi functions see Koornwinder's paper \cite{Koornwinder}). Therefore, the spherical functions $\varphi_{Y}(\lambda,r)$ are given by Jacobi functions in the following way 
\begin{align}
\varphi_{Y}(\lambda,r) &= \phi_{\lambda}^{(\frac{n-2}{2},\frac{-1}{2})}(r). 
\end{align}
Equivalently, in terms of the Gauss hypergeometric functions, we have 
\begin{equation}\label{reg-sol}
\varphi_{Y}(\lambda,r) = \, _2F_1\left(\frac{1}{2}\left((n-1)/2-i\lambda\right),\frac{1}{2}\left((n-1)/2+i\lambda\right);\frac{n}{2}; \, -\sinh^{2}(r)\right).
\end{equation}
It can also be seen \cite[p. 7]{Koornwinder} that for $\lambda\neq -i, -2i,...$, a second solution of (\ref{hyp-equa}) on $(0,+\infty)$ is given by
\begin{align}\label{sing-sol}
\varphi_{Y}(\lambda,r) \, =&\, (\cosh r)^{i\lambda-(n-1)/2} \, \times \nonumber\\
&\,  _2F_1 \left(\frac{1}{2}\left((n-1)/2-i\lambda\right),\frac{1}{2}\left((n-1)/2-i\lambda\right)+\frac{1}{2}; 1-i\lambda; \cosh^{-2}(r)\right).
\end{align}
It known that the resolvent kernel $R_Y(\lambda,r)$  is a multiple of the fundamental solution at infinity (\ref{sing-sol}) of (\ref{hyp-equa}) which reads as \cite{Agmon}
\begin{equation}\label{resol-hyp}
   R_Y(\lambda,r)= C_{n,\lambda} \left(\cosh r/2\right)^{2i\lambda-(n-1)} \, _2F_1 \left((n-1)/2-i\lambda, 1/2 -i\lambda; 1-2i\lambda; \cosh^{-2}(r/2)\right)
\end{equation}
where 
\begin{align}
C_{n,\lambda} = 2^{-(n-2i\lambda)} \pi^{-(n-1)/2} \Gamma((n-1)/2-i\lambda) / \Gamma(1-i\lambda). 
\end{align} 

In \cite{Davies}, it has been established a recurrence relations relating  the heat kernels $K(t,.)$ on the real hyperbolic spaces $H^{n-1}$ and $H^{n+1}$ of dimension $n-1$ and $n+1$ respectively,  given by the following  recurrence relation
\begin{align}
K_{n+1}(t,\rho) = -\frac{e^{(1-n)t}}{2\pi \sinh \rho} \, \frac{\partial}{\partial \rho} K_{n-1}(t,\rho), \quad \rho>0.
\end{align}
In what follows, we give an analogous of this result relating the real hyperbolic resolvent kernels on the real hyperbolic spaces of dimension $n$ and $n+2$ respectively. Then,  we prove that for odd dimensional real hyperbolic spaces the resolvent kernel is given in terms of elementary functions. Namely, our main result of this section is the following:
\begin{proposition}
Let $R_{n}(\lambda,r)$ be the resolvent kernel for real hyperbolic spaces $H^n$ of dimension $n$. Then

(i) The resolvent kernel in (\ref{resol-hyp}) can be written as 
\begin{align}\label{resol-hyp2}
   R_n(\lambda,r)= &\, C_n(\lambda) \left(\cosh r\right)^{i\lambda-(n-1)/2}\,  \times\nonumber\\ 
   & \, _2F_1 \left(\frac{1}{2}\left((n-1)/2-i\lambda\right),\frac{1}{2}\left((n-1)/2-i\lambda\right)+\frac{1}{2}; 1-i\lambda; \cosh^{-2}( r)\right)
\end{align}
where 
\begin{align}
C_n(\lambda) = (4\pi^{n/2})^{-1} \, \frac{\Gamma\left(\frac{1}{2}((n-1)/2-i\lambda)\right)\Gamma\left(\frac{1}{2}((n-1)/2-i\lambda)+1/2\right)}{\Gamma(1-i\lambda)}. 
\end{align}

(ii) The following recurrence relation hold: 
\begin{align}
\frac{-1}{2\pi \sinh r}\frac{\partial}{\partial r} \Bigr[R_n(\lambda,r)\Bigr] = R_{n+2}(\lambda,r).
\end{align}

(iii) The resolvent kernel for odd dimensional real hyperbolic space of dimension $2m + 1$ can be given in terms of elementary functions as follows
\begin{equation}\label{elem-resol}
   R_{2m+1}(\lambda,r) = C_m(\lambda) \,  \left(\frac{1}{\sinh r}\frac{\partial}{\partial r}\right)^m \left(e^{i r \lambda}\right),
\end{equation}
where \,  $C_m(\lambda) =(-1)^{m+1}/2i\lambda (2\pi)^m$.
\end{proposition}
\begin{proof}

Making use of the identity \cite[p. 50]{Magnus} 
\begin{align*}
  _2F_1\left(2a,c-1/2;2c-1;2 \sqrt z/(1+\sqrt z)\right) = (1+\sqrt z)^{2a} \, _2F_1(a,a+1/2;c; z)
\end{align*} 
and Legendre's duplication formula
\begin{align*}
\Gamma(2a)=(1/\sqrt \pi) \Gamma(a)\Gamma(a+1/2)2^{2a-1}
\end{align*} 
for $z=\cosh^{-2}(r)$, $a=((n-1)/2-i\lambda)/2$ and $c=1-i\lambda$, the resolvent kernel in (\ref{resol-hyp}) becomes
\begin{align*}
   R_n(\lambda,r)= &\, C_n(\lambda) \cosh^{i\lambda-(n-1)/2}(r) \, \times \nonumber\\ 
   &\, _2F_1 \left(\frac{1}{2}\left((n-1)/2-i\lambda\right),\frac{1}{2}\left((n-1)/2-i\lambda\right)+\frac{1}{2}; 1-i\lambda; \cosh^{-2}(r)\right)
\end{align*}
where 
\begin{align*}
C_n(\lambda) = (4\pi^{n/2})^{-1} \, \frac{\Gamma\left(\frac{1}{2}((n-1)/2-i\lambda)\right)\Gamma\left(\frac{1}{2}((n-1)/2-i\lambda)+1/2\right)}{\Gamma(1-i\lambda)}. 
\end{align*}
This proves the relation (i). The relation (ii) is obtained as follows. Recall that the real hyperbolic resolvent kernel of dimension $n$ is given in term of the hypergeometric function $_2F_1$ by \cite{Agmon} 
\begin{equation}\label{resol-hyp1}
   R_n(\lambda,r)= C_{n,\lambda} \,(\cosh r/2)^{2i\lambda-(n-1)} \, _2F_1 \left((n-1)/2-i\lambda, 1/2 -i\lambda; 1-2i\lambda; \cosh^{-2}(r/2)\right)
\end{equation}
where 
\begin{align*}
C_{n,\lambda} = 2^{-(n-2i\lambda)} \pi^{-(n-1)/2} \Gamma((n-1)/2-i\lambda) / \Gamma(1-i\lambda). 
\end{align*} 
By using the differential formula of the hypergeometric function $_2F_1$ \cite[p. 557]{Abramowitz}
\begin{align*}
\frac{\partial^k}{\partial z^k} \Bigr[z^{a+k-1}\,  _2F_1(a,b;c;z)\Bigr] = (a)_k \, z^{a-1}\,   _2F_1(a+k,b;c;z),
\end{align*}
where $(a)_k = a(a + 1) \dots (a + k -1)$ is the Pochhammer symbol, for $k=1$, we prove that the resolvent kernel $R_n(\lambda,r)$ in (\ref{resol-hyp1}) verifies the recurrence relation 
\begin{align}\label{recurrence-rel}
\frac{-1}{2\pi \sinh r}\frac{\partial}{\partial r} \Bigr[R_n(\lambda,r)\Bigr] = R_{n+2}(\lambda,r).
\end{align}
Finally, the relation (iii) is obtained easily from (ii). In fact, thanks to the recurrence formula (\ref{recurrence-rel}), we obtain the following expression for $R_{n}(\lambda,r)$ when $n$ is odd:
\begin{align}\label{rec-rel}
R_{2m+1}(\lambda,r) = \left(\frac{-1}{2\pi \sinh r}\frac{\partial}{\partial r}\right)^m \left[R_1(\lambda,r)\right] ,
\end{align}
where $R_{1}(\lambda,r)$ is the the resolvent kernel on the real hyperbolic space of one dimensional $H^1$. 
In addition, we observe that, for $n = 1$, the expression of the resolvent kernel $R_{n}(\lambda,r)$ given in (\ref{resol-hyp1}) reduces to 
\begin{equation*}
   R_1(\lambda,r)=  C_{1,\lambda} \left(\cosh r/2\right)^{2i\lambda} \, _2F_1 \left(-i\lambda, 1/2 -i\lambda; 1-2i\lambda; \cosh^{-2}(r/2)\right)   
\end{equation*}
where 
\begin{align*}
C_{1,\lambda} = 2^{2i\lambda-1}  \Gamma(-i\lambda) / \Gamma(1-i\lambda). 
\end{align*} 
Further, we use the well-known elementary expression of the hypergeometric function  given by \cite[p. 38]{Magnus}
\begin{equation*}
 _2F_1 \left(a, 1/2 +a; 2a+1; z^{2}\right) = 2^{2a} \left(1+\sqrt{1-z^2}\right)^{-2a} \,
\end{equation*}
for $a=-i\lambda$ and $z=1/\cosh(r/2)$, we arrive to the following identity 
\begin{equation*}
   R_{1}(\lambda,r) = \frac{-1}{2i\lambda} e^{i r \lambda}.
\end{equation*}
Replacing $R_{1}(\lambda,r)$ by its expression in the recurrence formula (\ref{rec-rel}), we obtain the announced relation iii). This ends the proof.
\end{proof}

\section{Gauss Hypergeometric functions  $_2F_1$}

The two expressions of the resolvent kernels in (\ref{resol-X}) and in (\ref{resol-hyp2}) are given in terms of the hypergeometric function $_2F_1$. For this reason, to find a proof of Theorem \ref{integral transform}, we should say something about hypergeometric functions.

\begin{equation}\label{F21}
_2F_1(a,b;c;z) = {\Gamma(c)\over \Gamma(a)\Gamma(b)} \sum_{k=0}^\infty {\Gamma(a+k)\Gamma(b+k)\over k!\Gamma(c+k)}  \, z^k,
\end{equation}
where $\Gamma(z)=\int_{0}^{\infty} t^{z-1}e^{-t}dt$ is Euler's Gamma-function. See \cite{Abramowitz}, \cite{Magnus} or \cite{Slater} for a general discussion of $_2F_1$'s and of more general series of this type.

In the following we prove some integral representation, connecting two hypergeometric functions. The corresponding kernel  function can be seen as the kernel of the Fourier-Jacobi or Olevskii index transform studied, for instance, in the book of S. B. Yakubovich \cite{Yakubovich}.

The most important properties of hypergeometric functions that we use all follow from the following integral
\begin{lemma}\label{lemma2}
For $ x > 1$ , $\Re \mu > 0$\ and $\Re b > 0$, we have
\begin{equation}\label{formule1} 
    x^{-b}\, _2F_1(a,b;c;x^{-1})={\Gamma(b+\mu)\over \Gamma(b)\Gamma(\mu)}\int_x^{\infty}y^{-b-\mu}(y-x)^{\mu-
1}\, _2F_1(a,b+\mu;c;y^{-1}) \,dy.
\end{equation}
\end{lemma}

\begin{proof} 
We use the well known series expansion for hypergeometric function: 
\begin{equation}\label{serieexp}
_2F_1(a,b+\mu;c;1/y) = {\Gamma(c)\over \Gamma(a)\Gamma(b+\mu)}\sum_{k=0}^\infty
{\Gamma(a+k)\Gamma(b+\mu+k)\over k!\Gamma(c+k)}1/y^k .
\end{equation}
Note that this series converges absolutely because of the above assumptions. Inserting $(\ref{serieexp})$ in the right hand side of the formula \eqref{formule1} and integrating term by term making the change of
variable $ y = xt^{-1}$ and use the Euler's formula for the beta function 
\begin{align*}
\int_{0}^{1} t^{a-1} (1-t)^{b-1} dt = \frac{\Gamma(a)\Gamma(b)}{\Gamma(a+b)},
\end{align*}
we obtain at once the desired result.
\end{proof}
Another integrals connecting hypergeometric functions with different parameters are the following
\begin{lemma}\label{lemma3}
For $ x > 1$ , $\Re \nu > 0$ and $\Re c > \Re \nu $, we have
\begin{equation*}
    x^{c-\nu}\, _2F_1(a,b;c;x^{-1})={\Gamma(c)\over \Gamma(c-\nu)\Gamma(\nu)}\int_x^{\infty}y^{-c}(y-x)^{\nu-
1}\, _2F_1(a,b;c-\nu;y^{-1}) \,dy.
\end{equation*}
\end{lemma}

This can be proved in the same way that Lemma \ref{lemma2} was proved, i.e. expand $_2F_1(a,b;c;y)$ in the series (\ref{F21}) and integrate term by term.

The other is a more complicated formula. In fact, in this lemma, we establish an integral identity between the Gauss hypergeometric functions $_2F_1(a,b;c;z)$ and $_2F_1(a+\mu,b+\nu;c;z)$, where $\mu$ and $\nu$ are given real numbers, given by the following lemma which will play a crucial role in the next section. Namely, we have
\begin{k-lemma} Let $a,b,\mu,\nu$ are complexes numbers such that $\Re a$, $\Re b$, $\Re \mu $, $\Re \nu>0$. Then for every $x>1$, the following identity holds
 \begin{align}\label{key lemma}
 x^{-b+\mu}\, _2F_1(a,b;c;x^{-1})&=\frac{\Gamma(a+\mu)\Gamma(b+\nu)}{\Gamma(a)\Gamma(b)\Gamma(\mu+\nu)}\, \times \nonumber\\
 & \int_x^{\infty}W_{ab\mu\nu}(x,y)y^{-b-\nu}\, _2F_1(a+\mu,b+\nu;c;y^{-1})\,dy
 \end{align}
where the kernel function $W_{ab\mu\nu}(x,y)$ is given by the following formula
\begin{equation}
W_{ab\mu\nu}(x,y)=(y-x)^{\mu+\nu-1}\, _2F_1(\mu,a-b+\mu;\mu+\nu;1-y/x).
\end{equation}
\end{k-lemma}

\begin{proof}
The proof will rely on the lemma \ref{lemma2}. In fact, to prove our key integral formula (\ref{key lemma}) we will apply the above integral identity (\ref{formule1}) and iterating it twice. That is, fixing $\mu$ and using the above lemma, we get
\begin{equation*}
 y^{-b}\, _2F_1(a+\mu,b;c;y^{-1})=\frac{\Gamma(b+\nu)}{\Gamma(b)\Gamma(\nu)} \int_y^{\infty} z^{-b-\nu}(z-y)^{\nu-1} \, _2F_1(a+\mu,b+\nu;c;z^{-1})\,dz.
 \end{equation*}
Therefore, multiplying both sides by $y^{b- a-\mu} (y -x)^{\mu-1}$ and integrating the both sides in $y$, we obtain
\begin{align*}
 x^{-a}\, _2F_1(a,b;c;x^{-1})= &\, \frac{\Gamma(a+\mu)\Gamma(b+\nu)}{\Gamma(a)\Gamma(b)\Gamma(\mu)\Gamma(\nu)} \,  \int_x^{\infty} y^{b-a-\mu}(y-x)^{\mu-1} \nonumber\\
& \times \,  \int_y^{\infty} z^{-b-\nu}(z-y)^{\nu-1} \, _2F_1(a+\mu,b+\nu;c;z^{-1})\,dz dy.
 \end{align*}
Note that $z > y > x > 1$ and by Fubini's theorem the integral
\begin{align*}
\int_x^{\infty} y^{b-a-\mu}(y-x)^{\mu-1} \left( \int_y^{\infty} z^{-b-\nu}(z-y)^{\nu-1}  \, _2F_1(a+\mu,b+\nu;c;z^{-1})\,dz \right) dy
\end{align*}
can be transformed to the integral
\begin{align}\label{int1}
\int_x^{\infty} z^{-b-\nu} \, _2F_1(a+\mu,b+\nu;c;z^{-1})\, \left( \int_x^{z} y^{b-a-\mu} (y-x)^{\mu-1} (z-y)^{\nu-1}  dy \right) dz.
\end{align}
Setting $y=(1-t)x + t z$, $t\in [0,1]$, then we have  
\begin{align*}
\int_x^{z} y^{b-a-\mu} (y-x)^{\mu-1} (z-y)^{\nu-1}\, dy = &  \, x^{b-a-\mu} (z-x)^{\mu+\nu-1} \, \times \nonumber\\
&  \, \int_0^{1} t^{\mu-1} (1-t)^{\nu-1} \left( 1-t(1-z/x)\right)^{-(a-b+\mu)}  dt.
\end{align*}
Using the integral representation of the hypergeometric function \cite[p. 558]{Abramowitz}
\begin{align*}
_2F_1(a',b';c';z') =  \frac{\Gamma(c')}{\Gamma(b')\Gamma(c'-b')} \int_0^{1} s^{b'-1} (1-s)^{c'-b'-1} \left( 1-s z'\right)^{-a'}  ds, \quad \Re(c')>\Re(b')>0,
\end{align*}
for $a' = a - b + \mu, b'= \mu, c'= \mu +  \nu$ and $z' = 1-z/x$, we get
\begin{align}\label{int2}
\int_x^{z} y^{b-a-\mu} (y-x)^{\mu-1} (z-y)^{\nu-1}\, dy = &  \, \frac{\Gamma(\mu)\Gamma(\nu)}{\Gamma(\mu+\nu)} \, x^{b-a-\mu} (z-x)^{\mu+\nu-1} \times \nonumber\\
& \, _2F_1(\mu,a - b + \mu; \mu +  \nu; 1-z/x).
\end{align}
Substituting (\ref{int2}) into the integral in (\ref{int1}), we obtain
\begin{equation*}
 x^{-b+\mu}\, _2F_1(a,b;c;x^{-1})=\frac{\Gamma(a+\mu)\Gamma(b+\nu)}{\Gamma(a)\Gamma(b)\Gamma(\mu+\nu)}\int_x^{\infty}W_{ab\mu\nu}(x,y)y^{-b-\nu}\, _2F_1(a+\mu,b+\nu;c;y^{-1})\,dy
 \end{equation*}
where  $W_{ab\mu\nu}(x,y)$ is given by 
\begin{equation*}
W_{ab\mu\nu}(x,y)=(y-x)^{\mu+\nu-1}\, _2F_1(\mu,a-b+\mu;\mu+\nu;1-y/x).
\end{equation*}
Hence the result of the key integral identity holds.
\end{proof}

\begin{remark} We shall notice that interchanging $a$ and $b$ together with $\mu$ and $\nu$ we get similarly
the following integral identity:
\begin{equation*} 
x^{-a+\nu}\, _2F_1(a,b;c;x^{-1})=\frac{\Gamma(a+\mu)\Gamma(b+\nu)}{\Gamma(a)\Gamma(b)\Gamma(\mu+\nu)}\int_x^{\infty}\widetilde{W}_{ab\mu\nu}(x,y)y^{-a-\mu}\, _2F_1(a+\mu,b+\nu;c;y^{-1})\,dy
\end{equation*}
where the kernel function $\widetilde{W}_{ab\mu\nu}(x,y)$ is given by the following formula
\begin{equation*}
\widetilde{W}_{ab\mu\nu}(x,y)=(y-x)^{\mu+\nu-1}\, _2F_1(\nu,a-b+\nu;\mu+\nu;1-y/x)=\left(\frac{y}{x}\right)^{a+\mu-(b+\nu)}W_{ab\mu\nu}(x,y).
\end{equation*}
\end{remark}

\section{Integral transform}

Here an integral transform with a kernel function for the resolvent kernel on $NA$ harmonic groups in terms of the resolvent kernel for an odd dimensional real hyperbolic space is obtained, which  are both  given in terms of the hypergeometric function $_2F_1(a,b;c;z)$.

%

For this, we deal with the application of the above key integral identities to the resolvent kernel function $R_X(\lambda;r)$.  
To do this properly, we recall its expression to be connected by the integral formula. 
It is well known that the resolvent kernel function of the shifted Laplacian $L_X+\sigma^2_X$ on $X$ is given in terms of the Gauss Hypergeometric functions $_2F_1(a,b;c;z)$ as follows
\begin{equation}\label{3}
   R_X(\lambda,r)=C_X(\lambda) \left(\cosh r\right)^{-\sigma+i\lambda} \, _2F_1 \left(\frac{\sigma-i\lambda}{2},\frac{\sigma-i\lambda}{2}-\beta; 1-i\lambda; \cosh^{-2}(r)\right)
\end{equation}
for $\Im m\lambda \geq 0$, $\beta=\frac{dim Z_X-1}{2}$ and the constant $C_X(\lambda)$ is given explicitly by
\begin{equation*}\label{constant}
C_X(\lambda)= \pi^{-(dim N_X+1)/2}\frac{\Gamma((\sigma-i\lambda)/2) \Gamma((\sigma-i\lambda)/2 -\beta)}{4\Gamma(1-i\lambda)}.
\end{equation*}

Then, our main result is the following:

\begin{theorem}\label{main-result} Let $X=NA$ be a harmonic group including rank one symmetric spaces of non compact type
(Real hyperbolic spaces are disregarded). Then the resolvent kernel $R_X(\lambda,r)$ of the Laplacian $L_X$ as given in (\ref{3}) can be expressed in terms of the resolvent kernel of odd dimensional hyperbolic spaces as follows
\begin{equation*}\label{integraltransform}
R_X(\lambda,r)=\int_r^\infty W_X(r,\rho)R_Y(\lambda,\rho)\sinh(\rho) d\rho
\end{equation*}
where the kernel $W_X(r,\rho)$ is independent of $\lambda\in\C$ and it is given by
\begin{equation*}\label{kernel}
W_X(r,\rho) = \frac{2\pi^{\frac{\dim Z_X}{2}}}{\Gamma(\dim Z_X/2)} \cosh^{(1-\dim Z_X)}(r)\left(\cosh^2 \rho-\cosh^2 r\right)^{(\dim Z_X-2)/2}
\end{equation*}
while
$R_Y(\lambda,\rho)$ is the resolvent kernel of some odd dimensional hyperbolic spaces $Y$ such that
$$\dim Y= \dim N_X+ \dim Z_X+1.$$
\end{theorem}

\begin{proof}
To establish this integral representation, we begin by representing the resolvent kernel $R_X(\lambda,r)$ after substitution $x=\cosh^{2}(r)$  as 
follows
\begin{align}
 R_X(\lambda,r)=: G_X(\lambda,x) =&\, C_X(\lambda) \, x^{-\frac{\sigma-i\lambda}{2}} \, _2F_1 \left(\frac{\sigma-i\lambda}{2},\frac{\sigma-i\lambda}{2}-\beta; 1-i\lambda; x^{-1}\right)
\end{align}
and appealing to the key integral formula (\ref{key lemma}) with $a= \frac{\sigma-i\lambda}{2}, b=\frac{\sigma-i\lambda}{2}-\beta$ and $c=1-i\lambda$, so that $\mu$ and $\nu$ are given respectively by $\mu=\frac{1}{2}$ and $\nu=\beta$, we have the following integral representation  
\begin{align}\label{G-X1}
G_X(\lambda,x) =&\,  C(\sigma,\beta,\lambda) \, x^{-\beta-\frac{1}{2}}\, \times \nonumber\\ 
 &\,\int_x^{\infty}K_{\beta}(x,y) \, y^{-\frac{\sigma-i\lambda}{2}}\, _2F_1\left(\frac{\sigma-i\lambda}{2} +\frac{1}{2}, \frac{\sigma-i\lambda}{2};1-i\lambda;y^{-1}\right)\,dy
\end{align}
where we denote by $K_{\beta}(x,y)$ the function of type 
\begin{equation*}\label{K-beta}
K_{\beta}(x,y)=(y-x)^{\beta-\frac{1}{2}}\, _2F_1\left(\frac{1}{2},\beta+\frac{1}{2};\beta+\frac{1}{2};1-y/x\right)
\end{equation*}
and where $$C(\sigma,\beta,\lambda)= \pi^{-(\dim N+1)/2} \frac{\Gamma(\frac{\sigma-i\lambda}{2}+\frac{1}{2})\Gamma(\frac{\sigma-i\lambda}{2})}{4 \Gamma(1-i\lambda)\Gamma(\beta+\frac{1}{2})}.$$

Using the identity \cite[p. 38]{Magnus}
\begin{align*}
(1+z)^{a} = \, _2F_1(-a,b;b;-z)
\end{align*}
for $a=-\frac{1}{2}, b=\beta+\frac{1}{2}$ and $z=-(1-y/x)$, we get
\begin{equation*}
K_{\beta}(x,y)=(y-x)^{\beta-\frac{1}{2}}\, \left(x/y\right)^{1/2}.
\end{equation*}
Then, the equation (\ref{G-X1}) becomes 
\begin{align}\label{G-X2}
 G_X(\lambda,x)&=\,  C(\sigma,\beta,\lambda) \, x^{-\beta} \, \times \nonumber\\ 
 &\,\int_x^{\infty} (y-x)^{\beta-\frac{1}{2}}\, y^{-\frac{\sigma-i\lambda}{2}-\frac{1}{2}}\, _2F_1\left(\frac{\sigma-i\lambda}{2} +\frac{1}{2}, \frac{\sigma-i\lambda}{2};1-i\lambda;y^{-1}\right)\,dy.
\end{align}
Considering the integral (\ref{G-X2}) by replacement of variable $x=\cosh^{2}(r)$, we arrive to the following identity 
\begin{align*}
R_X(\lambda,r) =\, &   2 C(\sigma,\beta,\lambda)\,\cosh^{-2\beta}(r)  \int_r^{\infty} \left(\cosh^2 \rho-\cosh^2 r\right)^{\beta-\frac{1}{2}} \, \sinh \rho \, \nonumber\\
 &\, \times \,  (\cosh \rho)^{-\sigma+i\lambda}\, _2F_1\left(\frac{\sigma-i\lambda}{2}, \frac{\sigma-i\lambda}{2}+\frac{1}{2};1-i\lambda;\cosh^{-2} \rho\right)\,  d\rho.
\end{align*}
Consequently one has the integral transform
\begin{align*}
 R_X(\lambda,r) &= \,\int_r^{\infty} W_X(r,\rho)\, R_Y(\lambda,\rho) \, \sinh \rho d\rho
\end{align*}
where we denote by $W_X(r,\rho)$ the kernel (independent of $\lambda\in\C$) given by
\begin{equation*}
W_X(r,\rho) = \frac{2\pi^{\frac{\dim Z_X}{2}}}{\Gamma(\dim Z_X/2)} \cosh^{(1-\dim Z_X)}(r)\left(\cosh^2 \rho-\cosh^2 r\right)^{(\dim Z_X-2)/2}
\end{equation*}
and where $R_Y(\lambda,\rho)$ is the resolvent kernel of some odd dimensional hyperbolic spaces $Y$ such that
$$\dim Y= \dim N_X+ \dim Z_X+1.$$

So we led  to the the desired result. Theorem \ref{main-result} is proved.
\end{proof}

\section{Concluding Remarks}


To finish this paper we have to mention that for complex and quaternionic hyperbolic spaces $G/K=NA$ there are some homogenous vector bundles $V_{\tau}$ over them and that the resolvent kernels of the Laplacian $L_{X,\tau}$ on $G/K=NA$ acting on sections of such vector bundles are given by:

\begin{equation}\label{3}
   R_{X,\tau}(\lambda,r)=C_X(\lambda) \left(\cosh r\right)^{-\sigma+i\lambda} \, _2F_1 \left(\frac{\sigma-i\lambda}{2},\frac{\sigma-i\lambda}{2}-\frac{\tau}{2}; 1-i\lambda; \cosh^{-2}(r)\right)
\end{equation}
for $\Im m\lambda \geq 0$, where the constant $C_X(\lambda)$ is given explicitly by
\begin{equation*}\label{constant}
C_X(\lambda)= \pi^{-(dim N_X+1)/2}\frac{\Gamma((\sigma-i\lambda)/2) \Gamma((\sigma-i\lambda)/2)}{4\Gamma(1-i\lambda)}.
\end{equation*}

Then there is also a general integral transform with kernel $W_{X,\tau}(r,\rho)$ that connect $R_{X,\tau}(\lambda,r)$ to the resolvent $R_Y(\lambda,\rho)$ as in Theorem \ref{main-result}. The detail is left for a forthcoming paper in relation with this subject.

We hope that we can use this explicit integral transform to solve other problems in the spherical analysis on the $NA$ harmonic groups.




\begin{thebibliography}{99}

\bibitem{Abramowitz} 
    Abramowitz M., Stegun I. A. Handbook of mathematical functions: with formulas, graphs, and mathematical tables. Courier Corporation. 1964; No. 55.
\bibitem{Agmon} 
    Agmon S. On the spectral theory of the Laplacian on noncompact hyperbolic manifolds. Journ\'{e}es \'{e}quations aux d\'{e}riv\'{e}es partielles, 1987; 1-16.
\bibitem{Anker}
    Anker J. P., Damek  E., Yacoub C. Spherical analysis on harmonic $ AN $ groups. Annali della Scuola Normale Superiore di Pisa-Classe di Scienze. 1996, 23(4): 643-679.
\bibitem{Berndt}
    Berndt J., Tricerri F.,  Vanhecke  L. Generalized Heisenberg groups and Damek-Ricci harmonic spaces. 1995
\bibitem{Damek-Ricci}
    Damek E., Ricci F. A class of nonsymmetric harmonic Riemannian spaces. Bull. Amer. Math. Soc. 1992, 27(1): 139-142.
\bibitem{Davies}
    Davies E. B., Mandouvalos N. Heat kernel bounds on hyperbolic space and Kleinian groups. Proceedings of the London Mathematical Society. 1988; 3(1): 182-208.
\bibitem{Helgason} 
Helgason S. Eigenspaces of the Laplacian; integral representations and irreducibility. Journal of Functional Analysis. 1974; 17(3): 328-353.
\bibitem{Kaplan}
    Kaplan A. Fundamental solutions for a class of hypoelliptic PDE generated by composition of quadratic forms. Transactions of the American Mathematical Society,. 1980; 258(1): 147-153.
\bibitem{Kaplan2}
    Kaplan  A.,  Ricci F. Harmonic analysis on groups of Heisenberg type. In Harmonic analysis. Springer Berlin Heidelberg. 1983: pp. 416-435.
\bibitem{Korany1}
    Kor$\acute{a}$nyi A. Geometric properties of Heisenberg-type groups. Advances in Mathematics. 1985; 56(1): 28-38.
\bibitem{Koornwinder}
    Koornwinder T. H. Jacobi functions and analysis on noncompact semisimple Lie groups. In Special functions: group theoretical aspects and applications. Springer Netherlands. 1984: pp. 1-85
\bibitem{Lichnerowicz}
    Lichnerowicz A. Sur les espaces riemanniens completement harmoniques. Bulletin de la Soci\'{e}t\'{e} Math\'{e}matique de France. 1944; 72: pp. 146-168.
\bibitem{Magnus} 
    Magnus W. Oberhettinger F, Soni  R. Formulas and Theorems for the Special Functions of Mathematical Physics. 1966
\bibitem{Wallach}     
Miatello R.,  Wallach N. R. The resolvent of the Laplacian on locally symmetric spaces. Journal of Differential Geometry, 1992; 36(3), 663-698.    
\bibitem{Will} 
  Miatello R.,  Will C. The residues of the resolvent on Damek-Ricci spaces. Proceedings of the American Mathematical Society. 2000; 128(4): 1221-1229.
\bibitem{Ricci}    
    Ricci F. The spherical transform on harmonic extensions of H-type groups. Politecnico di Torino. Dipartimento di Matematica. 1993.
\bibitem{Rouviere}
    Rouviere F. Espaces de Damek-Ricci, g\'{e}om\'{e}trie et analyse. Semin. Congr. 2003; 7: 45-100.
\bibitem{Slater} 
   Slater L. J. Generalized hypergeometric functions. 1966.
\bibitem{Yakubovich}  
   Yakubovich S. B. Index transforms. World Scientific, 1996.
\end{thebibliography}
\end{document}